\documentclass[12pt]{amsart}
\usepackage{amsmath, amsthm,mathabx}
\usepackage{enumitem}

\title{Decomposability of Nonnegative $r$-Potent Operators on $\mathcal{L}^2(\mathcal{X})$}

\author{Rashmi Sehgal Thukral}
\address{Department of Mathematics, Jesus and Mary College, University of Delhi, Chanakyapuri, New Delhi 110021}
\email{rashmi.sehgal@yahoo.co.in}

\author{Alka Marwaha}
\address{Department of Mathematics, Jesus and Mary College, University of Delhi, Chanakyapuri, New Delhi 110021}
\email{alkasamta@gmail.com}

\newtheoremstyle{theorem}
  {10pt}		  
  {10pt}  
  {\sl}  
  {\parindent}     
  {\bf}  
  {. }    
  { }    
  {}     
\theoremstyle{theorem}
\newtheorem{theorem}{Theorem}
\newtheorem{corollary}[theorem]{Corollary}

\newtheoremstyle{defi}
  {10pt}		  
  {10pt}  
  {\rm}  
  {\parindent}     
  {\bf}  
  {. }    
  { }    
  {}     
\theoremstyle{defi}
\newtheorem{definition}[theorem]{Definition}
\newtheorem{proposition}[theorem]{Proposition}

\newtheorem{lemma}[theorem]{Lemma}

\begin{document}


\maketitle
\paragraph{\bf Abstract:}
We investigate the decomposability of nonnegative compact $r$-potent operators on a separable Hilbert space $\mathcal{L}^2(\mathcal{X})$. We  provide a constructive algorithm to prove that basis functions of range spaces of nonnegative $r$-potent operators can be chosen to be all nonnegative and mutually orthogonal. We use this orthogonality to establish that nonnegative compact $r$-potent operators with range spaces of dimension strictly greater than $r-1$ are decomposable. \\


\section{Introduction}
An operator $\mathbf{A}$ is said to be {\it idempotent} if it satisfies $\mathbf{A}^2=\mathbf{A}$. An operator $\mathbf{A}$ on the Hilbert space $\mathcal{L}^2(\mathcal{X})$ is said to be {\it decomposable}  \cite{ColojoaraBook,RadjaviBook} if there exists a nontrivial standard subspace of $\mathcal{L}^2(\mathcal{X})$ that is invariant under $\mathbf{A}$. Marwaha \cite{MarwahaAndRadjavi} showed that  nonnegative {\it idempotent} operators of rank greater than one are decomposable in finite dimensions. Marwaha  \cite{MarwahaAndRadjaviInfiniteDimensions} further established that nonnegative idempotent operators with range spaces of dimension greater than one are decomposable in infinite dimensions. The results in \cite{MarwahaAndRadjavi} (finite-dimension case) for idempotent operators were generalised by Thukral and Marwaha to $r$-potent operators in \cite{ThukralAndMarwaha} (recall that an operator $\mathbf{A}$ is said to be $r$-potent \cite{McCloskeyFirstPaper} if $\mathbf{A}^r=\mathbf{A}$, where $r$ is a positive integer). Specifically, Thukral and Marwaha \cite{ThukralAndMarwaha} showed that all  nonnegative $r$-potent operators of rank greater than $r-1$ are decomposable in finite dimensions. A similar generalisation from idempotent operators to $r$-potent operators  in the infinite-dimension case \cite{MarwahaAndRadjaviInfiniteDimensions} however does not exist in the literature. In this paper, our main contribution is to provide this generalisation  of results in \cite{MarwahaAndRadjaviInfiniteDimensions} for idempotent operators to the $r$-potent case. It should be noted that proof of decomposability for $r$-potent operators in the finite-dimension case \cite{ThukralAndMarwaha} depends critically on the Perron-Frobenius Theorem \cite{CarlMeyer} for matrices. Since there is no direct analogue of Perron-Frobenius theorem in the infinite-dimension case, generalisation of  results in \cite{MarwahaAndRadjaviInfiniteDimensions} to the $r$-potent case is completely nontrivial. In fact, a substantial portion of this paper is devoted to developing new tools and techniques that are eventually deployed to prove the decomposability of $r$-potent operators in the infinite-dimension case. The rest of this  paper is organised as follows. Section~\ref{Sec:Preliminaries} contains the notations, definitions and preliminaries used in this paper. In Section~\ref{Sec:CharacterizationRangeSpace-r-PotentOperators}, we provide a novel constructive algorithm to show that the basis functions of the range space of a nonnegative $r$-potent operator can be chosen to be all nonnegative and mutually orthogonal. Finally, Section~\ref{Sec:DecomposabilityOf-r-PotentOperators} states the result on decomposability of $r$-potent operators and the corresponding proofs.


\section{Preliminaries}
\label{Sec:Preliminaries}
In this paper, $\mathcal{X}$ denotes a separable, locally compact Hausdorff space  and $\mu$ a Borel measure on $\mathcal{X}$. $\mathcal{L}^2(\mathcal{X})$ is the Hilbert space of (equivalence classes) complex-valued measurable and square integrable functions on $\mathcal{X}$. For sake of simplicity, we assume that $\mu(\mathcal{X})<\infty$. 

\begin{definition}\cite[p.~57]{JohnConwayBook}
A function $f\in \mathcal{L}^2(\mathcal{X})$ is said to be {\it nonnegative}, written as $f\geq 0$, if $\mu\{x\in\mathcal{X}:f(x)<0\}=0$.
\end{definition}
\begin{definition}\cite{RadjaviBook}
A {\it standard subspace} of $\mathcal{L}^2(\mathcal{X})$ is a norm-closed linear manifold in $\mathcal{L}^2(\mathcal{X})$ of the form
\begin{align*}
\mathcal{L}^2(U)=\{f\in\mathcal{L}^2(\mathcal{X}):f=0\ \text{a.e. on }U^c\}
\end{align*}
for some Borel subset $U$ of $\mathcal{X}$. This space is nontrivial if $\mu(U)\cdot\mu(U^c)>0$.
\end{definition}
\begin{definition}\cite{RadjaviBook}
An operator $\mathbf{A}$ on $\mathcal{L}^2(\mathcal{X})$ is said to be {\it decomposable} if there exists a nontrivial standard subspace of $\mathcal{L}^2(\mathcal{X})$ invariant under $\mathbf{A}$.  
\end{definition}
In this paper, we shall use the following equivalent definition of decomposability (provided as a proposition in \cite[p.~39]{MarwahaAndRadjaviInfiniteDimensions}):
\begin{definition}
\label{definitionOfDecomposabilityByAlkaMarwaha}
A nonnegative operator $\mathbf{A}$ on $\mathcal{L}^2(\mathcal{X})$ is decomposable if and only if there exists a Borel subset $U$ of $\mathcal{X}$ with $\mu(U)\cdot \mu(U^c)>0$ such that $
\langle\mathbf{A}\chi_U,\chi_{U^c}\rangle=0$, 
where $\chi_U=1$ on $\text{Supp}\ U$ and $\chi_U=0$ on $\text{Supp}\ U^c$ (Note that $U^c$ denotes the complement of $U$).
\end{definition}
\begin{definition}
Suppose $\mathcal{X}_1$ and $\mathcal{X}_2$ are Borel subsets of $\mathcal{X}$. An {\it operator} $\mathbf{A}$ from $\mathcal{L}^2(\mathcal{X}_1)$ to $\mathcal{L}^2(\mathcal{X}_2)$ is called {\it nonnegative} if $\mathbf{A}f\geq 0$ whenever $0\neq f\geq 0$ in $\mathcal{L}_2(\mathcal{X}_1)$. Similarly, $\mathbf{A}$ is called {\it positive} if $\mathbf{A}f> 0$ whenever $0\neq f\geq 0$ in $\mathcal{L}_2(\mathcal{X}_1)$.
\end{definition}
\begin{definition}
For any function $f$, we define the support of $f$ as $\text{Supp }f=\{x\in\mathcal{X}:f(x)\neq 0\}$. If $f$ is a member of $\mathcal{L}^2(\mathcal{X})$, then $\text{Supp }f$ is defined up to a set of measure zero. 
\end{definition}
\begin{definition}
Let $f,g$ be two nonnegative functions in $\mathcal{L}^2(\mathcal{X})$. Then, $\text{Supp }f$ and $\text{Supp }g$ are called nonoverlapping (or equivalently orthogonal) up to a set of measure zero if $\mu(\text{Supp }f\cap \text{Supp }g)=0$.		
\end{definition}
\begin{definition}
A function $f\in\mathcal{L}^2(\mathcal{X})$ is said to be mixed if 
\begin{align*}
f&=f^+-f^-, \\
 \text{where}\quad f^+&=\text{max}\{f,0\},\\
\text{and}\quad f^-&=\text{max}\{-f,0\},
\end{align*} 
are called the positive and negative parts of $f$, respectively, and the following two conditions are satisfied:
\begin{align*}
 \mu(\text{Supp}\ f^+)&>0 \\
 \mu(\text{Supp}\ f^-)&>0. 
\end{align*}
\end{definition}
\begin{definition}
An operator $\mathbf{A}$ is called {\it $r$-potent} if $\mathbf{A}^r=\mathbf{A}$, where $r$ is a positive integer.
\end{definition}
\begin{definition}\cite{KreyszigIntroductoryFunctionalAnalysisBook}
An operator $\mathbf{A}:\mathcal{L}^2(\mathcal{X}_1)\rightarrow \mathcal{L}^2(\mathcal{X}_2)$ is called a {\it compact linear operator} (or completely continuous linear operator) if $\mathbf{A}$ is linear and if for every bounded subset $\mathcal{M}$ of $\mathcal{L}^2(\mathcal{X}_1)$, the image $\mathbf{A}(\mathcal{M})$ is relatively compact, that is, the closure $\overline{\mathbf{A}(\mathcal{M})}$ is a compact subset of $\mathcal{L}^2(\mathcal{X}_2)$.
\end{definition}
\begin{proposition}\cite{MarwahaAndRadjaviInfiniteDimensions}
Any two nonnegative functions in $\mathcal{L}^2(\mathcal{X})$ are mutually orthogonal if and only if they have nonoverlapping support sets up to a set of measure zero. In other words, if $f,g\geq 0$ in $\mathcal{L}^2(\mathcal{X})$,  then $\langle f,g\rangle=0$ if and only if $\mu(\text{Supp }f\cap \text{Supp }g)=0$.
\end{proposition}
\begin{proposition}
If $f_1,f_2,\ldots,f_n$ are orthogonal functions in $\mathcal{L}^2(\mathcal{X})$, that is, $\langle f_i,f_j\rangle=0$, for all  $i\neq j$, then $f_1,\ldots,f_n$ are linearly independent.
\end{proposition}

We next state a set of four lemmas (Lemma~\ref{lemma:BF0ImpliesB0} and \ref{LemmaNonNegativeElementInKernal} are known in the literature \cite{AlkaMarwahaThesis} while Lemma~\ref{LemmaRealAndImaginary} and \ref{LemmaCompactOperatorsHaveFiniteDimensionalRange} are our contributions) along with the corresponding proofs that would allow us to set the context of the problem solved in this paper.
\begin{lemma}
\label{lemma:BF0ImpliesB0}
Let $\mathbf{B}:\mathcal{L}^2(\mathcal{X}_1)\rightarrow\mathcal{L}^2(\mathcal{X}_2)$ be a nonnegative operator such that $\mathbf{B}f_o=0$ for some $f_o>0$ in $\mathcal{L}^2(\mathcal{X}_1)$. Then $\mathbf{B}=0$.
\end{lemma}
\begin{proof}
We note that 
\begin{align*}
&\mathbf{B}\mathbf{f}_o=0\\
\Rightarrow& \langle\mathbf{B}f_o,g\rangle=0,\qquad \text{for all $g\geq 0$ in $\mathcal{L}^2(\mathcal{X}_2)$}\\
\Rightarrow& \langle f_o,\mathbf{B}^*g\rangle=0\qquad \text{for all $g\geq 0$ in $\mathcal{L}^2(\mathcal{X}_2)$}\\
\Rightarrow &\int_{\mathcal{X}}f_o(x)(\mathbf{B}^*g)(x)\mu(\mathrm{d}x)=0\\
\Rightarrow &f_o(x)(\mathbf{B}^*g)(x)=0\qquad \text{a.e. on $\mathcal{X}_1$}.
\end{align*}
However, $f_o(x)>0$ a.e., and therefore
\begin{align*}
(\mathbf{B}^*g)(x)&=0\qquad \text{a.e. for all $g\geq0$}\\
\Rightarrow \mathbf{B}^*g&=0 \qquad\text{for all $g\geq 0$}\\
\Rightarrow \mathbf{B}^*&=0 \\
\Rightarrow \mathbf{B}&=0.
\end{align*}
\end{proof}

\begin{lemma}
\label{LemmaNonNegativeElementInKernal}
If the kernel of a nonnegative operator $\mathbf{A}$ (denoted by Ker$(\mathbf{A})$) contains a nonzero nonnegative function, then $\mathbf{A}$ is decomposable. 
\end{lemma}
\begin{proof}
Let $h$ be a nonzero nonnegative function in $\mathrm{Ker}(\mathbf{A})$. Then,  $h>0$ on $\mathrm{Supp}\ h$,  implies that $\mathbf{A}h=0$ for $h>0$ in $\mathcal{L}^2(\mathrm{Supp}\ h)$. This however, due to Lemma~\ref{lemma:BF0ImpliesB0}, implies that $\mathbf{A}=0$ on $\mathcal{L}^2(\mathrm{Supp}\ h)$. Therefore, $\mathbf{A}$ is decomposable.
\end{proof}
The above lemma implies that any operator $\mathbf{A}$ which has a nonzero nonnegative function in its kernel is  decomposable. For the rest of this paper, we therefore assume that $\mathrm{Ker}(\mathbf{A})$ has no nonzero nonnegative functions. 

\begin{lemma}
\label{LemmaRealAndImaginary}
If a function $f$ belongs to the range space $R(\mathbf{A})$ of an $r$-potent operator $\mathbf{A}$, then both $\mathrm{Re}(f)$ and $\mathrm{Im}(f)$ also belong to $R(\mathbf{A})$. 
\end{lemma}
\begin{proof}
Since $f\in R(\mathbf{A})$ and $\mathbf{A}^{r-1}=\mathbf{I}$ on $R(\mathbf{A})$, we have
\begin{align}
\mathbf{A}^{r-1}f&=f\\
\label{Ar-1IntoRealAndImaginary}
\Rightarrow \mathbf{A}^{r-1}(\mathrm{Re}(f) + i\ \mathrm{Im}(f)) & = \mathrm{Re}(f) + i\ \mathrm{Im}(f)\\
\label{RealPartsEqual}
\Rightarrow \mathbf{A}^{r-1}(\mathrm{Re}(f))&=\mathrm{Re}(f)\quad \mathrm{and}\\
\label{ImaginaryPartsEqual}
\mathbf{A}^{r-1}(\mathrm{Im}(f))&=\mathrm{Im}(f), 
\end{align}
where Eqns.~(\ref{RealPartsEqual}) and (\ref{ImaginaryPartsEqual}) follow by comparing the real and imaginary parts of  Eqn.~(\ref{Ar-1IntoRealAndImaginary})  as $\mathbf{A}^{r-1}$ is a linear operator. However, Eqns. (\ref{RealPartsEqual}) and (\ref{ImaginaryPartsEqual}) also imply that both $\mathrm{Re}(f)$ and $\mathrm{Im}(f)$ belong to $R(\mathbf{A})$.
\end{proof}
Thanks to the above lemma, we shall restrict our focus in the rest of this paper to real functions only. 

\begin{lemma}
\label{LemmaCompactOperatorsHaveFiniteDimensionalRange}
A compact $r$-potent operator on a Hilbert space has a finite-dimensional range space.
\end{lemma}
\begin{proof}
We prove this lemma using the method of contradiction. Let $\mathbf{A}$ be a compact $r$-potent operator on a Hilbert space $\mathcal{L}^2(\mathcal{X})$ and let its range space $R(\mathbf{A})$ be infinite-dimensional. 
It is easy to see that $R(\mathbf{A})$ is a closed, and hence, complete subspace of $\mathcal{L}^2(\mathcal{X})$. 
 
Now, let $\{e_n\}_{n\in\mathbb{N}}$ be an orthonormal sequence in $R(\mathbf{A})$. Then, $\mathbf{A}^{r-1}e_n=e_n,\forall n$. Also, $\{e_n\}_{n\in\mathbb{N}}$ is a bounded sequence in $R(\mathbf{A})$. In what follows, we will show that $\{\mathbf{A}e_n\}_{n\in\mathbb{N}}$ does not have a convergent subsequence, contradicting the compactness of $\mathbf{A}$. 

We start by noting that since $\|e_n - e_m\| =\sqrt{2},\forall n\neq m$, the sequence $\{e_n\}_{n\in\mathbb{N}}$, which in turn is equal to $\{\mathbf{A}^{r-1}e_n\}_{n\in\mathbb{N}}$ due to $r$-potence, cannot have any convergent subsequence. Now, suppose the sequence $\{\mathbf{A}e_n\}_{n\in\mathbb{N}}$ has a convergent subsequence, say $\{\mathbf{A}e_{n_k}\}_{k\in\mathbb{N}}$. Then, since $R(\mathbf{A})$ is closed, $\{\mathbf{A}e_{n_k}\}_{n\in\mathbb{N}}$ must converge in $R(\mathbf{A})$. We can therefore write
\begin{align}
\mathbf{A}e_{n_k}&\rightarrow \mathbf{A}x\quad \mathrm{(say)}\\
\Rightarrow \mathbf{A}^{r-1}e_{n_k}&\rightarrow\mathbf{A}^{r-1}x\quad\text{(since $\mathbf{A}^{r-1}$ is continuous)}\\
\Rightarrow e_{n_k} &\rightarrow \mathbf{A}^{r-1}x\quad \text{(since $\mathbf{A}^{r-1}=\mathbf{I}$ on $R(\mathbf{A})$)},
\end{align}
which is a contradiction. It therefore follows that the sequence $\{\mathbf{A}e_n\}_{n\in\mathbb{N}}$ does not have any convergent subsequence, which is a contradiction to our assertion that $\mathbf{A}$ is compact.
\end{proof}
Since our focus in this paper is on nonnegative {\it compact} $r$-potent operators, we will assume throughout the paper that $r$-potent operators have finite-dimensional ranges. 

To summarise, Lemmas~\ref{LemmaNonNegativeElementInKernal}, \ref{LemmaRealAndImaginary} and \ref{LemmaCompactOperatorsHaveFiniteDimensionalRange} imply that for the rest of this paper, we need to focus only on operators with (1) no nonzero nonnegative functions in their kernels, (2) real basis functions and (3) finite dimensional range spaces (say of dimension $N$).


\section{Characterization of Range Spaces of $r$-Potent Operators}
\label{Sec:CharacterizationRangeSpace-r-PotentOperators}
We show in this section that it is always possible to choose the basis functions of the range space of a nonnegative compact $r$-potent operator to be all {\it nonnegative} and mutually orthogonal. The key word here is ``nonnegative" because if we simply require an orthogonal basis, we can obtain it using (say) Gram-Schmidt orthogonalization procedure. Typically, any orthogonal basis obtained using Gram-Schmidt procedure may however contain mixed functions. Obtaining a basis with all nonnegative yet orthogonal functions is completely nontrivial and is the main subject of this section. The rest of this section is organized as follows.

 We first present a result (stated as Lemma~\ref{LemmaPositiveAndNegativePartInRange} and Corollary~\ref{LemmaCanAlwaysGetAllElementsNonnegative}) to show that we can always construct a basis of $R(\mathbf{A})$ comprising of only nonnegative functions in $\mathcal{L}^2(\mathcal{X})$  from any given basis of $R(\mathbf{A})$. We then provide a set of three Lemmas (Lemma~ \ref{ConstructionOfMixedElementFromTwoLinIndElements},  \ref{LemmaInMixedVectorGenerationAllThreeElementsAreLinInd} and \ref{ConvertOverlappingElementsIntoNonOverlappingElements}) to demonstrate that if any pair of nonnegative basis functions have overlapping supports such that the measure of overlap is positive,  then we can always obtain an alternate basis where the two  elements with overlapping supports are replaced by certain alternate basis functions with nonoverlapping supports.   We finally present a stepwise algorithm (Theorem~\ref{TheoremOrthogonalityOfBasisElements}) that uses Lemma~\ref{ConvertOverlappingElementsIntoNonOverlappingElements} to show that we can systematically replace all  basis functions with overlapping supports in $R(\mathbf{A})$ by alternate nonnegative basis functions whose supports are nonoverlapping. 
 
 We start with  Lemma~\ref{LemmaPositiveAndNegativePartInRange} (generalization of \cite[Lemma~3.3]{ZhongPaper} from idempotent to $r$-potent case) and Corollary~\ref{LemmaCanAlwaysGetAllElementsNonnegative}.

\begin{lemma}
\label{LemmaPositiveAndNegativePartInRange}
If the kernel $\mathrm{Ker}(\mathbf{A})$ of a nonnegative $r$-potent operator $\mathbf{A}$ does not contain any  nonzero nonnegative function, then for every mixed function $f\in R(\mathbf{A})$, the positive and negative components (namely $f^+$ and $f^-$, respectively) of $f$ must also be elements of $R(\mathbf{A})$, that is, $f^+,f^-\in R(\mathbf{A})$. 
\end{lemma}
\begin{proof}
Since $f\in R(\mathbf{A})$, we have
\begin{align}
\mathbf{A}^{r-1}f&=f\\
\Rightarrow \mathbf{A}^{r-1}(f^+-f^-) & = f^+-f^-\\
\Rightarrow \mathbf{A}^{r-1}f^+-f^+&=\mathbf{A}^{r-1}f^- - f^-=h\quad\text{(say)}\\
\label{Ar-1f+Equalf++h}
\Rightarrow \mathbf{A}^{r-1}f^+&=f^+ +h\quad\text{and}\\
\label{Ar-1f-Equalf-+h}
\mathbf{A}^{r-1}f^-&=f^- +h.
\end{align}
Therefore, $f^++h$ and $f^-+h$ are both elements of $R(\mathbf{A})$. However,
\begin{align}
\label{Ahequalszero}
\mathbf{A}h&=\mathbf{A}(\mathbf{A}^{r-1}f^+-f^+)=\mathbf{A}^rf^+-\mathbf{A}f^+=\mathbf{A}f^+-\mathbf{A}f^+=0,
\end{align}
that is, $h$ is an element in $\text{Ker }(\mathbf{A})$. Since there is no nonzero nonnegative element in  $\text{Ker }(\mathbf{A})$ (recall Lemma~\ref{LemmaNonNegativeElementInKernal}), $h$ must be either zero or a mixed function. We next show that $h$ cannot be a mixed function. Specifically, we note that 
\begin{align}
f&=\mathbf{A}^{r-1}f\\
\label{Ar-1f+MinusAr-1f-Equalsf+Minusf-}
\Rightarrow f^+-f^-&=\mathbf{A}^{r-1}f^+ - \mathbf{A}^{r-1}f^-.
\end{align}
Since $ \mathbf{A}^{r-1}f^-\geq 0$ and $f^-=0$ on $\text{Supp }f^+$, Eqn.~(\ref{Ar-1f+MinusAr-1f-Equalsf+Minusf-}) implies
\begin{align}
\label{Ar-1f+GreaterThanf+OnSuppf+}
\mathbf{A}^{r-1}f^+&\geq f^+\quad\text{on Supp $f^+$}.
\end{align} 
Moreover, since $f^+=0$ on $(\text{Supp }f^+)^c$ while $\mathbf{A}^{r-1}f^+\geq 0$, we have
 \begin{align}
\label{Ar-1f+GreaterThanf+OnSuppf+Complement}
\mathbf{A}^{r-1}f^+&\geq f^+\quad\text{on $(\text{Supp }f^+)^c$}.
\end{align} 
Combining Eqn.~(\ref{Ar-1f+GreaterThanf+OnSuppf+}) and (\ref{Ar-1f+GreaterThanf+OnSuppf+Complement}) and noticing that 
\begin{align}
\text{Supp }f^+\cup(\text{Supp }f^+)^c&=\mathcal{X},
\end{align} 
we get
\begin{align}
\mathbf{A}^{r-1}f^+&\geq f^+\\
\Rightarrow h=\mathbf{A}^{r-1}f^+&- f^+\geq 0.
\end{align}
Therefore, $h$ cannot be a mixed function, and hence, $h$ must be zero. Consequently, both $f^+=\mathbf{A}^{r-1}f^+$ and $f^-=\mathbf{A}^{r-1}f^-$ (from Eqns~(\ref{Ar-1f+Equalf++h}) and (\ref{Ar-1f-Equalf-+h}), respectively) are elements in $R(\mathbf{A})$. 
\end{proof}

\begin{corollary}
\label{LemmaCanAlwaysGetAllElementsNonnegative}
Let $\{\mathbf{e}_1,\mathbf{e}_2,\ldots,\mathbf{e}_N\}$ be a set of basis functions in the $N$-dimensional range space $\mathcal{R}(\mathbf{A})$ of a nonnegative compact $r$-potent operator $\mathbf{A}$. Then, there must exist an alternate set of basis functions $\{\mathbf{e}_1',\mathbf{e}_2',\ldots,\mathbf{e}_N'\}$ such that $\mathbf{e}_j',\forall j$, are all nonnegative. 
\end{corollary}
\begin{proof}
The proof follows in a straightforward manner by noticing that all the mixed functions (with both positive and negative components) among $\{\mathbf{e}_1,\mathbf{e}_2,\ldots,\mathbf{e}_N\}$ can be replaced (due to Lemma~\ref{LemmaPositiveAndNegativePartInRange}) by their respective positive and negative components. By dropping any linearly dependent functions in the set so generated, we obtain the required basis $\{\mathbf{e}_1',\mathbf{e}_2',\ldots,\mathbf{e}_N'\}$. 
\end{proof}
Corollary~\ref{LemmaCanAlwaysGetAllElementsNonnegative} establishes that given any basis of an $r$-potent operator, we can always find an alternate basis where all basis functions are nonnegative (although they may still have overlapping supports). We will show in Theorem~\ref{TheoremOrthogonalityOfBasisElements} (stated later in this Section) that given a basis of an $r$-potent operator with all nonnegative elements, it is further possible to construct an alternate basis where all basis functions are {\it both nonnegative and mutually orthogonal} (that is, with nonoverlapping supports). For proof of Theorem~\ref{TheoremOrthogonalityOfBasisElements}, we require an important Lemma (stated later as Lemma~\ref{ConvertOverlappingElementsIntoNonOverlappingElements}), which shows that if any pair of basis functions has overlapping supports, then it can be replaced in the basis by alternate basis functions that have nonoverlapping supports. However, before proceeding to the proof of Lemma~\ref{ConvertOverlappingElementsIntoNonOverlappingElements}, we shall need the following two Lemmas as well.

\begin{lemma}\cite{AlkaMarwahaThesis}
\label{ConstructionOfMixedElementFromTwoLinIndElements}
Given two linearly independent functions $f,g\in\mathcal{L}^2(\mathcal{X})$, we can always construct a mixed function $u=u^+-u^-$  such that $\mathrm{Supp}\ u^+$ and $\mathrm{Supp}\ u^-$ are nonoverlapping up to a set of measure zero. Morevover, $u^+,u^-$ so generated would be linearly independent.
\end{lemma}
\begin{proof}
Consider the following two subsets of real numbers
\begin{align}
\mathcal{K}_1&=\{k:f-kg>0\}\nonumber\\
\mathcal{K}_2&=\{k:f-kg<0\}\nonumber,
\end{align}
and note that $\mathcal{K}_1\neq\phi$ because $0\in\mathcal{K}_1$ and $\mathcal{K}_2\neq \phi$ because $K_2=\phi\Rightarrow f>kg,\forall k\in\mathbb{R}$, which, of course, is not possible. 
Now, let $k_1=\mathrm{Supremum}\ \mathcal{K}_1$ and $k_2=\mathrm{Infimum}\ \mathcal{K}_2$. If $k_1$ and $k_2$ are finite, then $k_1\neq k_2$ as then $f$ and $g$ would be linearly dependent, which is not true. Therefore, we must have $k_1<k_2$ (Note that since $\mathcal{K}_1\neq\phi$, $k_1$ cannot be infinity when $k_2$ is infinite). If we chose any real scalar $p$ such that $k_1<p<k_2$, then $p\notin \mathcal{K}_1$ because $k_1<p$ and $p\notin \mathcal{K}_2$ because $p<k_2$. Therefore,  $f-pg\nsucc 0$ and $f-pg\nprec 0$, and hence, $f-pg$ is a mixed function. Let $f-pg=u^+-u^-$, where $u^+$ is the positive component and $u^-$ is the negative component of this mixed function. Then, $\mu(\mathrm{Supp}\ u^+\cap \mathrm{Supp}\ u^-)=0$, and consequently, $\{u^+,u^-\}$ are orthogonal and hence linearly independent.
\end{proof}

\begin{lemma}
\label{LemmaInMixedVectorGenerationAllThreeElementsAreLinInd}
Consider a mixed function $u=e_1-pe_2=u^+-u^-$ generated from linearly independent functions $e_1,e_2\in R(\mathbf{A})$. Then, each of the sets $\{e_1,u^+,u^-\}$ and $\{e_2,u^+,u^-\}$  is linearly independent.
\end{lemma}
\begin{proof}
We can divide the proof into the following cases:
\begin{description}
\item[Case 1]  The following three conditions hold:
\begin{align*}
\mu(\text{Supp}\ e_1\cap\text{Supp}\ u^+)&=0,\\
\mu(\text{Supp}\ e_1\cap\text{Supp}\ u^-)&=0,\\
\mu(\text{Supp}\ u^+\cap\text{Supp}\ u^-)&=0.
\end{align*}
In this case, $\{e_1,u^+\}$, $\{e_1,u^-\}$, $\{u^+,u^-\}$ are linearly independent. Therefore, $\{e_1,u^+,u^-\}$  is also linearly independent. 
\item[Case 2]  The following three conditions hold:
\begin{align*}
\mu(\text{Supp}\ e_1\cap\text{Supp}\ u^+)&=0,\\
\mu(\text{Supp}\ e_1\cap\text{Supp}\ u^-)&>0,\\
\mu(\text{Supp}\ u^+\cap\text{Supp}\ u^-)&=0.
\end{align*}
In this case, $\{e_1,u^+\}$ is linearly independent. Suppose $e_1=\alpha u^-$ for some $\alpha$, then $\alpha >0$ and $\alpha u^- - p e_2=u^+ - u^-$. Therefore, 
\begin{align*}
&(\alpha+1)u^- - u^+ = pe_2,\\
\Rightarrow &e_2= \frac{\alpha+1}{p}u^- - \frac{1}{p} u^+,
\end{align*}
which is a contradiction as $e_2\geq 0$. Therefore, $\{e_1,u^-\}$ is linearly independent, and consequently, $\{e_1,u^+,u^-\}$ is linearly independent. 
\item[Case 3] The following three conditions hold:
\begin{align*}
\mu(\text{Supp}\ e_1\cap\text{Supp}\ u^+)&>0,\\
\mu(\text{Supp}\ e_1\cap\text{Supp}\ u^-)&=0,\\
\mu(\text{Supp}\ u^+\cap\text{Supp}\ u^-)&=0.
\end{align*}
This case is similar to Case 2. 
\item[Case 4] The following three conditions hold:
\begin{align*}
\mu(\text{Supp}\ e_1\cap\text{Supp}\ u^+)&>0,\\
\mu(\text{Supp}\ e_1\cap\text{Supp}\ u^-)&>0,\\
\mu(\text{Supp}\ u^+\cap\text{Supp}\ u^-)&=0.
\end{align*}
Since $\mu(\text{Supp}\ e_1\cap\text{Supp}\ u^+)>0$, we have $e_1\neq 0$  but $u^-=0$ on $\text{Supp }e_1\cap \text{Supp }u^+$. In addition, since $\mu(\text{Supp}\ e_1\cap\text{Supp}\ u^-)>0$, we have $e_1\neq 0$ but $u^+=0$ on $\text{Supp }u^-\cap \text{Supp } e_1$.  Therefore, $\{e_1,u^-\}$ and $\{e_1,u^+\}$ are linearly independent.  Consequently, $\{e_1,u^+,u^-\}$ is linearly independent. 
\end{description}
Replacing $e_1$ by $e_2$ in the above statements, we get linear independence of $\{e_2,u^+,u^-\}$ as well. 
\end{proof}
We finally proceed to Lemma~\ref{ConvertOverlappingElementsIntoNonOverlappingElements}, which eventually forms the basis of our proof for Theorem~\ref{TheoremOrthogonalityOfBasisElements}. 

\begin{lemma}
\label{ConvertOverlappingElementsIntoNonOverlappingElements}
Let $\mathcal{E}=\{e_1,e_2,\ldots,e_N\}$ be a set of nonnegative basis functions of $R(\mathbf{A})$ such that 
\begin{enumerate}[label=(\alph*)]
\label{MainLemma}
\item \label{MainLemmaParta} There exist two distinct functions $e_r,e_s\in\mathcal{E}$ with identical supports, that is, $\text{Supp }e_r=\text{Supp }e_s$. Then, it is possible to get an alternate basis for $R(\mathbf{A})$ by replacing $e_r,e_s$ in $\mathcal{E}$ by two new functions $e_r',e_s'\in R(\mathbf{A})$ such that $\text{Supp } e_r'$ and $\text{Supp }e_s'$ have at most partial overlap.
\item \label{MainLemmaPartb} There exist two distinct functions $e_r,e_s \in\mathcal{E}$ with $\text{Supp }e_r\subseteq\text{Supp }e_s$. Then, we can construct at least two new basis functions $e_r',e_s'\in\mathcal{E}$ with orthogonal supports.
\item \label{MainLemmaPartbdash} There exist two distinct functions $e_r,e_s\in\mathcal{E}$ with their respective supports satisfying the following three conditions: $\mu(\text{Supp}\ e_r\cap\text{Supp}\ e_s)>0$, $\mu(\text{Supp}\ e_r\cap\text{Supp}\ e_s^c)>0$ and $\mu(\text{Supp}\ e_r^c\cap\text{Supp}\ e_s)>0$. Further, let $e_r$ and $e_s$ be  linearly dependent on $\text{Supp}\ e_r\cap\text{Supp}\ e_s$. Then, we can construct three orthogonal nonnegative basis functions in $R(\mathbf{A})$ from $e_r$ and $e_s$. 
\item \label{MainLemmaPartc} There exist two distinct functions $e_r,e_s\in\mathcal{E}$ with their respective supports satisfying the following three conditions: $\mu(\text{Supp}\ e_r\cap\text{Supp}\ e_s)>0$, $\mu(\text{Supp}\ e_r\cap\text{Supp}\ e_s^c)>0$ and $\mu(\text{Supp}\ e_r^c\cap\text{Supp}\ e_s)>0$. Further, let $e_r$ and $e_s$ be linearly independent on $\text{Supp}\ e_r\cap\text{Supp}\ e_s$. Then, we can construct four orthogonal nonnegative basis functions in $R(\mathbf{A})$ from $e_r$ and $e_s$. 
\end{enumerate} 
\end{lemma}
\begin{proof} 
We prove the four parts of the Lemma separately:
\begin{enumerate}[label=(\alph*)]
\item
Since the set $\{e_r,e_s\}$ is linearly independent, we can construct a mixed function $u$ from $e_r$ and $e_s$ using Lemma~\ref{ConstructionOfMixedElementFromTwoLinIndElements}. Specifically, let $e_r-pe_s=u^+-u^-$, where, as previously, $u^+$ and $u^-$ are the positive and negative components of $u$. Then, by Lemma~\ref{LemmaPositiveAndNegativePartInRange}, we get $\{u^+,u^-\}\subseteq R(\mathbf{A})$. Moreover, if both $u^+$ and $u^-$ do not lie in the linear span of the remaining basis functions $\mathcal{E}-\{e_r,e_s\}$, then we can replace $\{e_r,e_s\}$ in $\mathcal{E}$ by $\{u^+,u^-\}$ and the resulting set would be the required new basis. However, if either of $u^+$ or $u^-$ lies in the linear span of the remaining basis functions $\mathcal{E}-\{e_r,e_s\}$, then we can replace $\{e_r,e_s\}$ in $\mathcal{E}$ by the linearly independent  set $\{e_r,u^+,u^-\}$ (due to Lemma~\ref{LemmaInMixedVectorGenerationAllThreeElementsAreLinInd}), so that the resulting set $\mathcal{E}'=\mathcal{E}-\{e_s\}+\{u^+,u^-\}$ spans the entire $R(\mathbf{A})$. However, $\mathcal{E}'$ contains $N+1$ functions and therefore cannot be a basis of $R(\mathbf{A})$, which is $N$-dimensional. Removing any of the functions in $\mathcal{E}'$ that is linearly dependent on the remaining functions in $\mathcal{E}'$ would yield the required new basis. Finally, note that $u^+$ and $u^-$ cannot both be linearly dependent on the remaining functions $\mathcal{E}-\{e_r,e_s\}$ because in this case, $e_s$ would be linearly dependent on $\mathcal{E}-\{e_s\}$, contradicting the premise that $\mathcal{E}$ is a set of basis functions.
\item
We shall further split this part into two cases:
\begin{description}
\item[Case 1] {\it $\{e_r,e_s\}$ is linearly dependent on Supp $e_r$}\\
Let $e_s=f+\alpha e_r$, where $f\geq 0$ and $\text{Supp}\ f\cap \text{Supp}\ e_r=\phi$. Since $f=e_s-\alpha e_r$ and both $e_s,e_r\in R(\mathbf{A})$, we must have $f\in R(\mathbf{A})$. Therefore, $\{f,e_r\}\in R(\mathbf{A})$ are nonnegative with orthogonal support sets and are hence linearly independent. Consequently, we can enter $f$ and remove $e_s$  from $\mathcal{E}$ to obtain a new set of basis functions. Please note that $f\notin \mathcal{E}-\{e_s\}$ as otherwise $e_s$ would lie in the linear span of $\mathcal{E}-\{e_s\}$, contradicting our premise that $\mathcal{E}$ is a set of basis functions. 
\item[Case 2] 
{\it $\{e_r,e_s\}$ is linearly independent on Supp $e_r$}\\
Let $e_s=f+g$, where $f\geq0$, $g\geq 0$, $\text{Supp}\ g= \text{Supp}\ e_r$ and $\text{Supp}\ f= \text{Supp}\ e_s\cap (\text{Supp}\ e_r)^c$.  Then, 
\begin{align*}
e_s-pe_r=f+g-pe_r&=\underbrace{f}_{>0}+\underbrace{(g-pe_r)}_{<0}\\
&=u^+-u^-,
\end{align*}
where $p$ is chosen such that $g-pe_r<0$ (Note that we can always choose such a $p$ because otherwise $g-ke_r\geq 0$, for all $k$, which is not possible). Thus, $u^+=f$ and $u^-=-(g-pe_r)$ are the positive and negative components of the mixed function $e_s-pe_r$. Consequently, from Lemma~\ref{LemmaPositiveAndNegativePartInRange}, $u^+=f\in R(\mathbf{A})$ and $u^-=-(g-pe_r)\in R(\mathbf{A})$. However, $f\in R(\mathbf{A})$ implies $g\in R(\mathbf{A})$ (since $e_s=f+g\in R(\mathbf{A})$). Moreover, since $\{e_r,g\}$ is linearly independent, $e_r-p'g=w^+-w^-$ is a mixed vector in $R(\mathbf{A})$ (for some $p'$), and therefore, $\{w^+,w^-,f\}\subseteq R(\mathbf{A})$ are linearly independent functions which are nonnegative and orthogonal. 
\end{description} 

\item
Let 
\begin{align*}
e_r&=f_r+h_r,\quad \text{where}\\
f_r&=e_r\big|_{\text{Supp}\ e_r-(\text{Supp}\ e_r\cap\text{Supp}\ e_s)}\quad \text{and}\\
h_r&=e_r\big|_{ \text{Supp}\ e_r\cap\text{Supp}\ e_s}.
\end{align*}
Further, let
\begin{align*}
e_s&=f_s+h_s\\
&=f_s+\alpha h_r\quad \text{where}\\
h_s&=\alpha h_r\big|_{\text{Supp}\ e_r\cap\text{Supp}\ e_s}\quad \text{and}\\
f_s&=e_s\big|_{\text{Supp}\ e_s-(\text{Supp}\ e_r\cap\text{Supp}\ e_s)}.
\end{align*}
Then, consider
\begin{align*}
e_r-\frac{1}{\alpha}e_s&=f_r+h_r-\frac{f_s}{\alpha}-h_r\\
&=\underbrace{f_r}_{\text{+ve part}} - \underbrace{\frac{f_s}{\alpha}}_{\text{-ve part}},
\end{align*}
which, due to Lemma~\ref{LemmaPositiveAndNegativePartInRange}, implies $f_r,f_s\in R(\mathbf{A})$, and consequently, $h_r\in R(\mathbf{A})$. Therefore, $f_r,f_s,h_r\in R(\mathbf{A})$ are nonnegative functions with orthogonal support sets.  

\item
Let 
\begin{align*}
e_r&=f_r+h_r,\quad\text{and}\\
e_s&=f_s+h_s\quad \text{where}\\
f_r&=e_r\big|_{\text{Supp}\ e_r-(\text{Supp}\ e_r\cap\text{Supp}\ e_s)}\\
h_r&=e_r\big|_{ \text{Supp}\ e_r\cap\text{Supp}\ e_s}\\
f_s&=e_s\big|_{\text{Supp}\ e_s-(\text{Supp}\ e_r\cap\text{Supp}\ e_s)}\quad {\text{and}}\\
h_s&=e_s\big|_{ \text{Supp}\ e_r\cap\text{Supp}\ e_s}.
\end{align*}
Further, let $\alpha$ be a real scalar satisfying $h_r>\alpha h_s$ (note that such an $\alpha$ would always exist). Then, 
\begin{align*}
e_r-\alpha e_s&=f_r+h_r -\alpha f_s - \alpha h_s\\
&=\underbrace{f_r+h_r-\alpha h_s}_{\text{+ve part}} -\underbrace{\alpha f_s}_{\text{-ve part}}
\end{align*} 
is a mixed function, and therefore $\{f_r+h_r-\alpha h_s,f_s\}\subseteq R(\mathbf{A})$. However, $f_s\in R(\mathbf{A})$ implies (since $e_s\in R(\mathbf{A})$) that $h_s=e_s-f_s\in R(\mathbf{A})$. Similarly, by considering $e_s-\alpha e_r$, we can show that $f_r,h_r\in R(\mathbf{A})$. Finally, since $\{h_r,h_s\}$ is linearly independent, we can construct a mixed function $x=h_r-\beta h_s=x^+-x^-$ (where $x^+$ and $x^-$ are the positive and negative components, respectively, of $x$), for some $\beta$, so that $\{x^+,x^-,f_r,f_s\}$ are four nonnegative orthogonal functions  in $R(\mathbf{A})$ which are all obtained from $\{e_r,e_s\}$.
\end{enumerate}
\end{proof}
Note that we can apply the analysis of the cases \ref{MainLemma}\ref{MainLemmaParta} and \ref{MainLemma}\ref{MainLemmaPartb}, respectively,  to (a) any pair of basis elements in $R(\mathbf{A})$ with identical supports, and (b) any pair of basis elements in $R(\mathbf{A})$ where support of one of the element is contained in the support of the other element, so that, eventually, every pairwise overlap of the supports of basis functions is partial in nature. The cases of pairwise partial overlap among supports of basis functions has been covered in parts \ref{MainLemma}\ref{MainLemmaPartbdash} and \ref{MainLemma}\ref{MainLemmaPartc}. We next state the main result of this section as  Theorem~\ref{TheoremOrthogonalityOfBasisElements} and present an algorithm that  deploys Lemmas~\ref{MainLemma}\ref{MainLemmaPartbdash} and \ref{MainLemma}\ref{MainLemmaPartc} in an iterative manner to systematically replace every pair of basis functions with partially overlapping supports in the basis by certain new set of basis functions that are nonnegative and have nonoverlapping supports, so that eventually all the basis functions are nonnegative and have nonoverlapping supports.  

\begin{theorem}
\label{TheoremOrthogonalityOfBasisElements}
Let $\mathcal{E}=\{\mathbf{e}_1,\mathbf{e}_2,\ldots,\mathbf{e}_N\}$ be a set of nonnegative basis functions in the range space $\mathcal{R}(\mathbf{A})$ of a nonnegative compact $r$-potent operator $\mathbf{A}$. Then, there must exist an alternate set of basis functions $\mathcal{E}'=\{\mathbf{e}_1',\mathbf{e}_2',\ldots,\mathbf{e}_N'\}$ such that $\mathbf{e}_j',$ for all $j$, are nonnegative and have nonoverlapping supports. 
\end{theorem}
\begin{proof}
We prove this theorem by providing a constructive algorithm to obtain a basis with all elements nonnegative and having  nonoverlapping supports from any given basis with all nonnegative elements:
\section*{Algorithm}
Given the set of basis function $\mathcal{E}$, we create two groups of functions: $\mathcal{E}_{Orth}$ containing functions that are all nonnegative and orthogonal, and $\mathcal{E}_{NonOrth}$ containing the remaining functions that may have partially overlapping supports. Naturally, the sum of the number of functions in $\mathcal{E}_{Orth}$ and $\mathcal{E}_{NonOrth}$ would be $N$ (that is, the dimension of $R(\mathbf{A}))$. The key idea behind our algorithm is to randomly select functions from $\mathcal{E}_{NonOrth}$, one at a time, and convert them into function(s) in $R(\mathbf{A})$ that have nonoverlapping supports with all functions already in $\mathcal{E}_{Orth}$ and can therefore be included in $\mathcal{E}_{Orth}$. This conversion is applied in an iterative fashion so that eventually all functions are nonnegative and have mutually orthogonal supports. Specifically, our algorithm consists of following steps:
\begin{description}
\item[Step 1]  
We begin with a single basis function, say $e_1$, in $\mathcal{E}_{Orth}$. We initialize two counters
\begin{description}
\item[{\it Counter 1}] Number of functions in $\mathcal{E}_{Orth}$
\item[{\it Counter 2}] Number of functions in $\mathcal{E}_{NonOrth}$ 
\end{description}
with {\it Counter 1} $=1$ and {\it Counter 2} $=N-1$.\\

\item[Step 2] 
We randomly pick one basis function, say $e_2$, from $\mathcal{E}_{NonOrth}$. Then, $\text{Supp}\ e_2$ is either orthogonal  to $\text{Supp}\ e_1$ or partially overlaps with $\text{Supp}\ e_1$. We shall consider this scenario in three   mutually exclusive and exhaustive cases:
\begin{description}
\item[Case I]{\it $e_1$ and $e_2$ are orthogonal}\\
In this case, $e_2$ is removed from $\mathcal{E}_{NonOrth}$ and inserted in $\mathcal{E}_{Orth}$. 
\item[Case II]{\it $\{e_1,e_2\}$ is linearly dependent on $\text{Supp}\ e_1\cap \text{Supp}\ e_2$}\\
By Lemma~\ref{MainLemma}\ref{MainLemmaPartbdash}, we can get three orthogonal nonnegative functions, say $f_1,f_2,f_3$, in $R(\mathbf{A})$, from $e_1$ and $e_2$. Therefore, we can replace $e_1$ in $\mathcal{E}_{Orth}$ by $\{f_1,f_2,f_3\}$, and drop $e_2$ from $\mathcal{E}_{NonOrth}$. In doing so however, the total number of functions in $\mathcal{E}_{Orth}$ and $\mathcal{E}_{NonOrth}$ becomes $3 + (N-2)=N+1$. Since the range space $R(\mathbf{A})$ is $N$-dimensional, one of the functions in $\mathcal{E}_{NonOrth}$ must be linearly dependent on the remaining $N$ functions. We shall drop this function from $\mathcal{E}_{NonOrth}$ so that $\mathcal{E}_{Orth}$ and $\mathcal{E}_{NonOrth}$ are together a set of basis functions.
\item[Case III]{\it $\{e_1,e_2\}$ is linearly independent on $\text{Supp}\ e_1\cap\text{Supp}\ e_2$}\\
By Lemma~\ref{MainLemma}\ref{MainLemmaPartc}, we can get four nonnegative orthogonal functions, say $g_1,g_2,g_3,g_4$ in $R(\mathbf{A})$ from $e_1$ and $e_2$. Therefore, we can replace $e_1$  in $\mathcal{E}_{Orth}$ by $\{g_1,g_2,g_3,g_4\}$, and drop $e_2$ from $\mathcal{E}_{NonOrth}$. In doing so however, the total number of functions in $\mathcal{E}_{Orth}$ and $\mathcal{E}_{NonOrth}$ becomes $4 + (N-2)=N+2$. Since the range space $R(\mathbf{A})$ is $N$-dimensional, there would be two functions in $\mathcal{E}_{NonOrth}$ that are linearly dependent on the remaining $N$ functions. We shall drop these two functions from $\mathcal{E}_{NonOrth}$ so that $\mathcal{E}_{Orth}$ and $\mathcal{E}_{NonOrth}$ are together a set of basis functions.
\end{description}
Our analysis for Cases I, II and III show that {\it Counter 1} would now become either 2, 3 or 4, that is, {\it Counter 1} increases by at least one in Step 2 and, correspondingly, {\it Counter 2} reduces by at least one. 

\item[Step 3] We select another function, say $e_3$, from $\mathcal{E}_{NonOrth}$. Then, $e_3$ would be either orthogonal to all functions in $\mathcal{E}_{Orth}$ or $\text{Supp}\ e_3$ would partially overlap with one or more of the functions in $\mathcal{E}_{Orth}$. If $e_3$ is orthogonal to all functions in $\mathcal{E}_{Orth}$, it can be directly removed from $\mathcal{E}_{NonOrth}$ and inserted in $\mathcal{E}_{Orth}$, thereby increasing {\it Counter 1} by $1$ and reducing {\it Counter 2} by $1$. If, on the other hand, $\text{Supp}\ e_3$  overlaps with one of the functions in $\mathcal{E}_{Orth}$, say $h$, then we shall have the following two cases:
\begin{description}
\item[Case I] {\it $\{h,e_3\}$ is linearly dependent on $\text{Supp}\ h\cap \text{Supp}\ e_3$}\\
By Lemma~\ref{MainLemma}\ref{MainLemmaPartbdash}, we can get three orthogonal basis functions from $h$ and $e_3$, namely, $f_1',f_2',f_3'$ such that
\begin{align*}
 \text{Supp}\ f_1' &\subseteq \text{Supp}\ h\cap (\text{Supp}\ e_3)^c\\
 \text{Supp}\ f_2' &\subseteq\text{Supp}\ e_3\cap (\text{Supp}\ h)^c\\
 \text{Supp}\ f_3' &\subseteq\text{Supp}\ e_3\cap \text{Supp} h,
 \end{align*}
 respectively. 
\item[Case II]{\it $\{h,e_3\}$ is linearly independent on $\text{Supp}\ h\cap \text{Supp}\ e_3$}\\
As $\{h,e_3\}$ is linearly independent on $\text{Supp}\ h\cap \text{Supp}\ e_3$, 
by Lemma~\ref{MainLemma}\ref{MainLemmaPartc}, 
we would get four orthogonal basis functions, namely, $g_1',g_2',g_3',g_4'$ such that 
\begin{align*}
\text{Supp}\ g_1'  &\subseteq \text{Supp}\ h\cap (\text{Supp}\ e_3)^c\\
\text{Supp}\ g_2' &\subseteq\text{Supp}\ e_3\cap (\text{Supp}\ h)^c\\
\text{Supp}\ g_3' &\subseteq\text{Supp}\ e_3\cap \text{Supp} h\\
\text{Supp}\ g_4' &\subseteq\text{Supp}\ e_3\cap \text{Supp} h,
\end{align*}
respectively.
\end{description}
If $\text{Supp}\ e_3$ overlaps with another function (say $h'$) in  $\mathcal{E}_{Orth}$, then it would be on $\text{Supp}\ e_3\cap (\text{Supp} h)^c$ (recall $\{h,h'\}$ have orthogonal supports). 
In addition, in Case I above, $(\text{Supp}\ h'$ $\cap \text{Supp}\ e_3)\subseteq \text{Supp}\ f_2'$ and therefore we can repeat the same argument as above for $\{h',f_2'\}$ rather than $\{h,e_3\}$ to obtain three new orthogonal functions, and so on. In Case II above,   
We have $(\text{Supp}\ h'\cap \text{Supp}\ e_3)\subseteq \text{Supp}\ g_2'$ and therefore we can repeat the same argument as above for $\{h',g_2'\}$ rather than $\{h,e_3\}$ to obtain four new orthogonal functions, and so on. The above process is repeated till all the overlaps between $e_3$ and functions of $\mathcal{E}_{Orth}$ are eliminated. The resulting orthogonal functions are aggregated in the new $\mathcal{E}_{Orth}$ and any set of functions in $\mathcal{E}_{NonOrth}$ that are all linearly dependent on the remaining $N$ functions in $\mathcal{E}_{Orth}$ and $\mathcal{E}_{NonOrth}$, are dropped.  
\end{description}
By performing the above step, we get a new $\mathcal{E}_{Orth}$ where number of functions in definitely greater than the number of functions in $\mathcal{E}_{Orth}$ obtained after Step 2. That is, due to Step 3, {\it Counter 1} increases by at least one, {\it Counter 2} decreases by the same number as the increase in {\it Counter 1} and the functions in $\mathcal{E}_{Orth}\cup \mathcal{E}_{NonOrth}$ still constitute a basis.

Finally, note that repeating Step 3 above for every remaining function in $\mathcal{E}_{NonOrth}$, one at a  time, we will eventually get {\it Counter 1} $= N$ (recall that number of basis functions cannot exceed the dimension of $R(\mathbf{A})$, which is $N$). The $\mathcal{E}_{Orth}$ so obtained would be the required basis of $R(\mathbf{A})$ with all functions nonnegative and mutually orthogonal. Therefore, the proof for Theorem~\ref{TheoremOrthogonalityOfBasisElements} is complete.
\end{proof}

\section{Decomposability of $r$-Potent Operators}
\label{Sec:DecomposabilityOf-r-PotentOperators}
Our main result in this section is stated as the following theorem: 
\begin{theorem}
If there exists a basis $\{\mathbf{e}_1,\mathbf{e}_2,\ldots,\mathbf{e}_N\}$ in the range space $R(\mathbf{A})$ of a nonnegative compact $r$-potent operator $\mathbf{A}$ with $r\leq N$ such that $\mathbf{e}_j$, for all $j$, are nonnegative and have non-overlapping supports, then $\mathbf{A}$ must be decomposable over some support set $\mathcal{U}$. 
\end{theorem}
\begin{proof}
We will first prove the above theorem for the case where $r=3$ and will then generalize the proof to $r>3$ case.
\section*{Case: $r=3$}
We start by noting that $\mathbf{A}e_i\in R(\mathbf{A})$, for $i=1,\ldots,N$. Therefore,
\begin{align}
\mathbf{A}e_i=\alpha_1^i e_1 +\alpha_2^ie_2+\cdots + \alpha_N^ie_N,
\end{align}
for some $\alpha_1^i,\ldots,\alpha_N^i\geq 0$ (since $\mathbf{A}$ is a nonnegative operator and $e_j$, for all $j$, are nonnegative orthogonal). 

We next claim that $\mathbf{A}e_i$ must lie in the linear span of exactly one $e_j (j=1,\ldots,N)$, that is, there must exist only one $j_o$ such that $\alpha_j^i\neq 0,$ for $j=j_o$, and $\alpha_j^i=0$, for all $j\neq j_o$. We can prove this claim as follows:

First, consider that
\begin{align}
\label{EqnTripotent-ei-InTermsOfOtherElements}
e_i=\mathbf{A}^2e_i&=\alpha_1^i\mathbf{A}e_1 +\alpha_2^i\mathbf{A}e_2+\cdots+\alpha_N^i\mathbf{A}e_N.
\end{align}
Now suppose that two coefficients, say $\alpha_p^i$ and $\alpha_q^i$, are both greater than zero, and rewrite Eqn.~(\ref{EqnTripotent-ei-InTermsOfOtherElements}) as
\begin{align}
\label{EqnTripotent-ei-InTermsOf-er-es-andOtherElements}
e_i=\mathbf{A}^2e_i&=\alpha_p^i\mathbf{A}e_p + \alpha_q^i\mathbf{A}e_q + \sum_{j\neq p,q}\alpha_j^i\mathbf{A}e_j.
\end{align}
However, since both $\mathbf{A}e_p,\mathbf{A}e_q\in R(\mathbf{A})$, we can write
\begin{align}
\label{EqnTripotentWriting-er-InTermsOfOtherElements}
\mathbf{A}e_p&=\alpha_1^pe_1+\alpha_2^pe_2+\cdots+\alpha_N^pe_N\\
\label{EqnTripotentWriting-es-InTermsOfOtherElements}
\mathbf{A}e_q&=\alpha_1^qe_1+\alpha_2^qe_2+\cdots+\alpha_N^qe_N.
\end{align}
Substituting Eqns.~(\ref{EqnTripotentWriting-er-InTermsOfOtherElements}) and (\ref{EqnTripotentWriting-es-InTermsOfOtherElements}) into Eqn.~(\ref{EqnTripotent-ei-InTermsOf-er-es-andOtherElements}), we get
\begin{align}
e_i&=\alpha_p^i(\alpha_1^pe_1+\alpha_2^pe_2+\cdots+\alpha_N^pe_N)\nonumber\\
&+\alpha_q^i(\alpha_1^qe_1+\alpha_2^qe_2+\cdots+\alpha_N^qe_N)\nonumber\\
&+\sum_{j\neq p,q}\alpha_j^i\mathbf{A}e_j.
\end{align}
As $\{e_1,e_2,\ldots,e_N\}$ is linearly independent and $\alpha_j^i$ are all nonnegative, we must have 
\begin{align}
\alpha_1^p&=\alpha_2^p=\cdots=\alpha_{i-1}^p=\alpha_{i+1}^p=\cdots=\alpha_N^p=0\quad\text{and}\nonumber\\
\alpha_1^q&=\alpha_2^q=\cdots=\alpha_{i-1}^q=\alpha_{i+1}^q=\cdots=\alpha_N^q=0
\end{align}
which implies
\begin{align}
\mathbf{A}e_p&=\alpha_i^pe_i\quad\text{and}\nonumber\\
\mathbf{A}e_q&=\alpha_i^qe_i
\end{align}
where both $\alpha_i^p,\alpha_i^q>0$ (otherwise, $\mathbf{A}$ would not be $3$-potent). Therefore, 
\begin{align}
e_p&=\mathbf{A}^2e_p\\
&=\alpha_i^p\mathbf{A}e_i\\
&=\alpha_i^p(\alpha_1^ie_1+\alpha_2^ie_2+\cdots +\alpha_N^ie_N)\\
&=\alpha_i^p\left(\alpha_p^ie_p+\alpha_q^ie_q+\sum_j\alpha_j^ie_j\right),
\end{align}
which, again using linearly independence of $\{e_1,e_2,\ldots,e_N\}$ implies $\alpha_i^p\alpha_q^i=0$. However, this is a contradiction to our assumption that both $\alpha_i^p,\alpha_i^q>0$, and therefore, our claim that $\mathbf{A}e_i$ must lie in the linear span of exactly one $e_j$ is proved. In other words, we must have a unique $j$ for which
\begin{align}
\mathbf{A}e_i&=\alpha_j^ie_j, \quad \text{where}\quad \alpha_j^i>0.
\end{align}

Note that if $\text{dim}\ R(\mathbf{A})=2$, then $\{e_1,e_2\}$ is a nonnegative orthogonal basis of $R(\mathbf{A})$. Therefore, due to the above proved claim, $\mathbf{A}e_1$ is equal to either $\alpha e_1$ or $\alpha'e_2$, for some $\alpha,\alpha'>0$. Similarly, $\mathbf{A}e_2$ is equal to either $\beta e_2$ or $\beta' e_1$, for some $\beta,\beta'>0$. 

But
\begin{align}
\mathbf{A}e_1&=\alpha e_1\\
\Rightarrow e_1&=\mathbf{A}^2e_1=\alpha\mathbf{A}e_1=\alpha^2e_1\\
\Rightarrow \alpha&=1 \quad\text{(since $\alpha >0$)},
\end{align}
which yields
\begin{align}
\mathbf{A}e_1&=e_1.
\end{align} 
In addition, $\mathbf{A}e_1=\alpha e_1$  also implies $\mathbf{A}e_2=\beta e_2$ (since $\mathbf{A}e_2=\beta' e_1$ would imply $e_2=\mathbf{A}^2e_2=\beta'\mathbf{A}e_1=\alpha\beta' e_1$, which would be a contradiction), which in turn implies
\begin{align}
e_2&=\mathbf{A}^2e_2=\beta\mathbf{A}e_2=\beta^2 e_2,
\end{align}
which, along with the condition $\beta >0$, implies $\mathbf{A}e_2=e_2$. That is, $\mathbf{A}=\mathbf{I}$ on $R(\mathbf{A})$. In other words, $\mathbf{A}$ is a 2-potent (or idempotent) operator which is already known to be decomposable. Therefore, we may assume that $\mathbf{A}e_1\neq \alpha e_1$ and $\mathbf{A}e_2\neq \beta e_2$. 

Consequently, 
\begin{align}
\mathbf{A}e_1&=\alpha'e_2\\
\Rightarrow e_1&=\mathbf{A}^2e_1=\alpha'\mathbf{A}e_2=\alpha'\beta'e_1,\quad \text{and}\\
 \mathbf{A}e_2&=\beta'e_1\\
\Rightarrow e_2&=\mathbf{A}^2e_2=\beta'\mathbf{A}e_1=\beta'\alpha'e_2.
\end{align}
Therefore, applying $\mathbf{A}$ to $e_1$ yields $e_2$ and applying $\mathbf{A}$ to $e_2$, in turn, yields $e_1$. This implies that if we need decomposability, we must have $\text{dim}\ R(\mathbf{A})>2$. 

In case $\text{dim}\ R(\mathbf{A})>2$, let us consider $e_1+\mathbf{A}e_1$. Then,
\begin{align}
\mu(\underbrace{\text{Supp}\ (e_1+\mathbf{A}e_1)}_{\mathcal{U}})&>0.
\end{align}
Also, since $\text{dim}\ R(\mathbf{A})>2$,
\begin{align}
\mu(\mathcal{U}^c)&>0.
\end{align}
Further, 
\begin{align}
\langle e_1+\mathbf{A}e_1,\chi_{\mathcal{U}^c}\rangle &=0\\
\Rightarrow \langle\mathbf{A}e_1+\mathbf{A}^2e_1,\chi_{\mathcal{U}^c}\rangle &=0\\
\Rightarrow \langle e_1+\mathbf{A}e_1,\mathbf{A}^*\chi_{\mathcal{U}^c}\rangle &=0\\
\label{EqnIntegralOf-e1-Ae1-EqualsZero}
\Rightarrow \int_{\mathcal{X}}(e_1+\mathbf{A}e_1)(x)\mathbf{A}^*\chi_{\mathcal{U}^c}(x)\mu(\mathrm{d}x) &=0
\end{align}
However, since $(e_1+\mathbf{A}e_1)(x)>0$ on $\mathcal{U}$, Eqn~(\ref{EqnIntegralOf-e1-Ae1-EqualsZero}) implies that
\begin{align}
\mathbf{A}^*\chi_{\mathcal{U}^c}(x)&=0\quad \text{a.e. on }\mathcal{U}\\
\Rightarrow \langle\mathbf{A}^*\chi_{\mathcal{U}^c},\chi_{\mathcal{U}}\rangle &=0\\
\Rightarrow \langle\chi_{\mathcal{U}^c},\mathbf{A}\chi_{\mathcal{U}}\rangle &=0,
\end{align}
and therefore, as per Defn.~\ref{definitionOfDecomposabilityByAlkaMarwaha}, the $3$-potent operator $\mathbf{A}$ is decomposable. We next generalise the above proof to the $r>3$ case.

\section*{Case: $r>3$}
We start by noting that $\mathbf{A}e_i\in R(\mathbf{A})$, for $i=1,2,\ldots,N$, and therefore
\begin{align}
\mathbf{A}e_i&=\alpha_1^ie_1+\alpha_2^ie_2+\cdots+\alpha_N^ie_N,
\end{align}
for some $\alpha_1^i,\ldots,\alpha_N^i\geq 0$, where at least one of $\alpha_1^i,\ldots,\alpha_N^i$ is strictly greater than zero or else $\mathbf{A}$ would not be an $r$-potent operator. 

We next claim that $\mathbf{A}e_i$ lies in the linear span of exactly one $e_j$ ($j=1,2,\ldots,N$), that is, there exists a unique $j_o$ such that $\alpha_j^i>0$, for $j=j_o$ and $\alpha_j^i=0$, for all $j\neq j_o$. We prove this claim by the method of contradiction. We assume that there exists two coefficients $\alpha_p^i$ and $\alpha_q^i$ such that both $\alpha_p^i,\alpha_q^i>0$. Then, 
\begin{align}
\label{EqnExpansionOfAei}
\mathbf{A}e_i&=\alpha_p^ie_p+\alpha_q^ie_q+\sum_{j\neq p,q}\alpha_j^ie_j
\end{align}
and therefore
\begin{align}
\label{EqnExpansionOfASquareei}
\mathbf{A}^2e_i&=\alpha_p^i\mathbf{A}e_p+\alpha_q^i\mathbf{A}e_q+\sum_{j\neq p,q}\alpha_j^i\mathbf{A}e_j\nonumber\\
&=\alpha_p^i(p_1e_p'+\text{NNOT}) +\alpha_q^i(q_1e_q'+\text{NNOT})+\text{NNOT},
\end{align}
where $e_p'$ and $e_q'$ are some basis functions whose coefficients are nonzero in the expansions of $\mathbf{A}e_p$ and $\mathbf{A}e_q$, respectively (recall that at least one of the coefficients must be strictly positive in every expansion else $\mathbf{A}$ would not be $r$-potent), $p_1,q_1$ are those strictly positive coefficients, and NNOT (NonNegative Other Terms) is a general term  introduced for the sake of convenience to denote any sum of nonnegative functions whose exact value is not relevant for further analysis in this proof. 

In a manner similar to expansions of $\mathbf{A}e_i$ and $\mathbf{A}^2e_i$ in  Eqns.~(\ref{EqnExpansionOfAei}) and (\ref{EqnExpansionOfASquareei}), respectively, we can write expansions of $\mathbf{A}^3e_i,\cdots,\mathbf{A}^{r-2}e_i$ as
\begin{align}
\mathbf{A}^3e_i&=\alpha_p^i\mathbf{A}^2e_p+\alpha_q^i\mathbf{A}^2e_q+\sum_{j\neq p,q}\alpha_j^i\mathbf{A}^2e_j\nonumber\\
&=\alpha_p^i(p_1\mathbf{A}e_p'+\text{NNOT})+\alpha_q^i(q_1\mathbf{A}e_q'+\text{NNOT}) +\text{NNOT}\nonumber\\
&=\alpha_p^i(p_1(p_2e_p''+\text{NNOT})+\text{NNOT}) \nonumber\\
&\qquad+\alpha_q^i(q_1(q_2e_q''+\text{NNOT})+\text{NNOT})+\text{NNOT} \\
&\qquad\qquad\qquad\qquad \Large{\vdots}\nonumber
\end{align}
\begin{align}
\label{EqnPartExpansionOfAr-2}
\mathbf{A}^{r-2}e_i&=\alpha_p^i\mathbf{A}^{r-3}e_p+\alpha_q^i\mathbf{A}^{r-3}e_q+\text{NNOT}\\
\label{EqnFundamentalExpansionOfAr-2}
&=\alpha_p^ip_1p_2\ldots p_{r-3}e_p +\alpha_q^iq_1q_2\ldots q_{r-3}e_q+\text{NNOT}\quad
\end{align}
where $\alpha_p^i,p_1,p_2,\ldots ,p_{r-3}$ and $\alpha_q^i,q_1,q_2,\ldots ,q_{r-3}$ are all chosen to be positive (note that we can do so because for any $e_j$, $\mathbf{A}e_j$ must have at least one coefficient $\alpha_j^k>0$ else $\mathbf{A}$ would not be $r$-potent).
Applying $\mathbf{A}$ to Eqn.~(\ref{EqnFundamentalExpansionOfAr-2}) yields
\begin{align}
e_i&=\mathbf{A}^{r-1}e_i\\
\label{EqnFundamentalExpansionOfAr-1}
&=\alpha_p^ip_1p_2\ldots p_{r-3}\mathbf{A}e_p +\alpha_q^iq_1q_2\ldots q_{r-3}\mathbf{A}e_q+\text{NNOT}
\end{align}
In Eqn~(\ref{EqnFundamentalExpansionOfAr-1}), at least one of the coefficients in expansion of both $\mathbf{A}e_p$ and $\mathbf{A}e_q$ is nonzero. Without loss of generality, let the term corresponding to the nonzero coefficient of $\mathbf{A}e_p$ be $\beta_me_m$ and the term corresponding to nonzero coefficient of $\mathbf{A}e_q$ by $\gamma_ne_n$. Then,
\begin{align}
e_i&=(\alpha_o^ip_1p_2\ldots p_{r-3}\beta_m)e_m +(\alpha_q^iq_1q_2\ldots q_{n-3}\gamma_n)e_n+\text{NNOT}.
\end{align} 
As the coefficients of both $e_m$ and $e_n$ in the above equation are strictly greater than zero, we must have both $m=i$ and $n=i$ (otherwise, due to linear independence of $e_1,\ldots ,e_N$, all the coefficients would be zero). 
Therefore, 
\begin{align}
\mathbf{A}e_p&=\beta_ie_i
\end{align} 
since $\mathbf{A}e_p$ will not involve any other nonzero term because then, as earlier, there will be a contradiction to our assertion that $e_j$ for all $j$, are linearly independent. By a similar argument as above, we can show that 
\begin{align}
\mathbf{A}e_q&=\gamma_ie_i.
\end{align}

Since Eqns.~(\ref{EqnPartExpansionOfAr-2})  and (\ref{EqnFundamentalExpansionOfAr-2}) sets
\begin{align}
\mathbf{A}^{r-3}e_p&=p_1p_2\ldots p_{r-3}e_p +\text{NNOT},\quad\text{and}\\
\mathbf{A}^{r-3}e_q&=q_1q_2\ldots q_{r-3}e_q +\text{NNOT},
\end{align}
we can now write
\begin{align}
\label{EqnFinalExpansionAr-2eo}
\mathbf{A}^{r-2}e_p&=p_1p_2\ldots p_{r-3}\mathbf{A}e_p +\text{NNOT}\nonumber\\
&=p_1p_2\ldots p_{r-3}\beta_ie_i+\text{NNOT}\quad \text{and}\\
\label{EqnFinalExpansionAr-2es}
\mathbf{A}^{r-2}e_q&=q_1q_2\ldots q_{r-3}\mathbf{A}e_q+\text{NNOT}\nonumber\\
&=q_1q_2\ldots q_{r-3}\gamma_ie_i+\text{NNOT},
\end{align}
which further implies
\begin{align}
\label{EqnFinale_oExpansion}
e_p&=\mathbf{A}^{r-1}e_p=p_1p_2\ldots p_{r-3}\beta_i\mathbf{A}e_i +\text{NNOT}\quad\text{and}\\
e_q&=\mathbf{A}^{r-1}e_q=q_1q_2\ldots q_{r-3}\gamma_i\mathbf{A}e_i+\text{NNOT}.
\end{align}
Expanding Eqn.~(\ref{EqnFinale_oExpansion}) for $\mathbf{A}e_i$ using Eqn.~(\ref{EqnExpansionOfAei}), we get
\begin{align}
e_p&=p_1p_2\ldots p_{r-3}\beta_i\left(\alpha_p^ie_p+\alpha_q^ie_q+\sum_{j\neq p,q}\alpha_j^ie_j\right)
\end{align}
which is a contradiction to the linear independence of $e_p$ and $e_q$ as $p_1p_2\ldots p_{r-3}\beta_i\alpha_q^i >0$. Therefore, the following must hold
\begin{align}
\mathbf{A}e_i&=\alpha_j^ie_j,
\end{align}
for a unique $j$ such that $\alpha_j^i>0$. This proves our claim that $\mathbf{A}e_i$ must lie in the linear span of exactly one $e_j$. 

Since the claim holds for all $e_i$, it follows that $\mathbf{A}e_i$ lies in the linear span of $e_j$, $\mathbf{A}^2e_i=\alpha_j^i\mathbf{A}e_j$ lies in the linear span of $e_k$ (for some $k$), and so on. However, since $\mathbf{A}^{r-1}e_i=e_i$, if we keep successively applying $\mathbf{A}$ to a basis function $e_i$, $e_i$ must repeat in a cyclic fashion (at least once for every $r-1$ applications of $\mathbf{A}$).  Please note that for at least one $e_i$, the frequency of this cyclic repetition should be $r-1$, else $\mathbf{A}$ would not be $r$-potent. 

Now, if $\text{dim}\ R(\mathbf{A})=r-1$, then $\{e_1,e_2,\ldots,e_{r-1}\}$ would be a nonnegative orthogonal basis for $R(\mathbf{A})$ and $\mathbf{A}$ would not be decomposable. Therefore, in order to have decomposability of an $r$-potent (but not $k$-potent for $k<r$) operator, the $\text{dim}\ R(\mathbf{A})$ must be greater than $r-1$. 

If $\text{dim}\ R(\mathbf{A})$ is indeed greater than $r-1$, then 
\begin{align}
\mathcal{L}^2\left(\underbrace{\text{Supp}\ (e_i+\mathbf{A}e_i+\mathbf{A}^2e_i+\cdots +\mathbf{A}^{r-2}e_i)}_{\mathcal{U}}\right)\nonumber
\end{align}
would be the required decomposing space with $\mu(\mathcal{U})\cdot\mu(\mathcal{U}^c)>0$. Specifically,
\begin{align}
\langle e_i+\mathbf{A}e_i+\mathbf{A}^2e_i+\cdots +\mathbf{A}^{r-2}e_i,\chi_{\mathcal{U}^c}\rangle &=0\\
\Rightarrow \langle\mathbf{A}^{r-1}e_i+\mathbf{A}e_i+\mathbf{A}^2e_i+\cdots +\mathbf{A}^{r-2}e_i,\chi_{\mathcal{U}^c}\rangle &=0\\
\Rightarrow \langle e_i+\mathbf{A}e_i+\mathbf{A}^2e_i+\cdots +\mathbf{A}^{r-2}e_i,\mathbf{A}^*\chi_{\mathcal{U}^c}\rangle &=0\\
\Rightarrow\int_{\mathcal{X}}\left(e_i+\mathbf{A}e_i+\cdots +\mathbf{A}^{r-2}e_i\right)(x)\mathbf{A}^*\chi_{\mathcal{U}^c}(x)\mu(\mathrm{d}x)&=0
\end{align}
which implies
\begin{align}
\mathbf{A}^*\chi_{\mathcal{U}^c}(x)&=0\quad \text{a.e. on $\mathcal{U}$}\\
\Rightarrow \langle\mathbf{A}^*\chi_{\mathcal{U}^c},\chi_{\mathcal{U}}\rangle &=0\\
\Rightarrow \langle\chi_{\mathcal{U}^c},\mathbf{A}\chi_{\mathcal{U}}\rangle &=0,
\end{align}
which, by Defn.~\ref{definitionOfDecomposabilityByAlkaMarwaha}, implies that the nonnegative $r$-potent operator $\mathbf{A}$ is decomposable. 
\end{proof}


\end{document}